\documentclass[12pt,a4paper]{article}
\usepackage[cp1251]{inputenc}%\usepackage[utf8]{inputenc}
\usepackage[hidelinks]{hyperref}
\usepackage[english,russian]{babel}
\usepackage{amsmath,amsthm,dsfont,pifont,amsfonts,MnSymbol}
\usepackage{MPatr_Style_English_2018}
\DeclareMathOperator{\rt}{\mathsf{root}}
\DeclareMathOperator{\lh}{\mathsf{length}}
\DeclareMathOperator{\nbhds}{\mathsf{nbhds}}
\DeclareMathOperator{\spa}{\mathsf{span}}
\DeclareMathOperator{\shoot}{\mathsf{shoot}}
\newcommand{\inn}{\mathop{\dot{\in}}}

\DeclareMathOperator{\rr}{\mathsf{res}}

\author{Mikhail Patrakeev\footnote{Krasovskii Institute of Mathematics and Mechanics of UB RAS, 620990, 16 Sofia Kovalevskaya street, Yekaterinburg, Russia and Ural Federal University, 620002, 19 Mira street, Yekaterinburg, Russia; patrakeev@mail.ru}\footnote{
This work is supported by Act 211 Government of the Russian Federation,  contract no.\,02.A03.21.0006}
}

\title{A space with a Lusin $\pi\mathsurround=0pt$-base \\ whose square has no Lusin $\pi\mathsurround=0pt$-base\footnote{2010 Mathematics Subject Classification\textup{:} Primary 54{E}99; Secondary 54{H}05. Keywords\textup{:} Lusin pi-base, pi-tree, the Baire space, the Sorgenfrey line, Souslin scheme, Lusin scheme, product of topological spaces}
}
\date{}

\begin{document}%\linespread{1.3}
\hyphenation{no-n-in-cre-a-s-ing tran-si-ti-ve par-ti-al}
\renewcommand{\proofname}{\textup{\textbf{Proof}}}
\renewcommand{\abstractname}{\textup{Abstract}}
\renewcommand{\refname}{\textup{References}}
\mathsurround=2pt%\overfullrule5pt
\maketitle

\begin{abstract}
We construct a space ${X}$ that has a Lusin $\pi\mathsurround=0pt$-base and such that ${X}^{2}$ has no Lusin $\pi\mathsurround=0pt$-base.
\end{abstract}

\section{Introduction}
\label{section0}

The class of topological spaces with a Lusin $\pi\mathsurround=0pt$-base, see Definition~\ref{def.luz.p.base}, was introduced in~\cite{MPatr1}; this class equals the class of spaces with a $\pi\mathsurround=0pt$-tree~\cite[Remark~11]{MPatr2}.
In this paper we build a space with a Lusin $\pi\mathsurround=0pt$-base whose square has no Lusin $\pi\mathsurround=0pt$-base, see Theorem~\ref{teo}.

The Baire space $\omega^\omega,$ the Sorgenfrey line $\mathcal{S},$ and the irrational Sorgenfrey line $\mathcal{I}$ have a Lusin $\pi\mathsurround=0pt$-base~\cite{MPatr1,MPatr3}, and all at most countable products of $\omega^\omega,$ $\hspace{1pt}\mathcal{S},$ and $\mathcal{I}$ also have a Lusin $\pi\mathsurround=0pt$-base~\cite{MPatr3}.
If a space ${X}$ has a Lusin $\pi\mathsurround=0pt$-base, then the products ${X}\times\omega^\omega,$ ${X}\times\mathcal{S},$ and ${X}\times\mathcal{S}^{\omega}$ have a Lusin $\pi\mathsurround=0pt$-base~\cite{MPatr3},
and also ${X}\setminus{F}$ has a Lusin $\pi\mathsurround=0pt$-base whenever ${F}\subseteq{X}$ is a $\mathsurround=0pt\sigma$-compact~\cite{MPatr2} (but a dense open subset of $X$ can be without a Lusin $\pi\mathsurround=0pt$-base).

If a space ${X}$  has a Lusin $\pi\mathsurround=0pt$-base, then it can be mapped onto $\omega^\omega$ by a continuous one-to-one map and also ${X}\hspace{-1pt}$ can be mapped onto an arbitrary nonempty Polish space by a continuous open map~\cite{MPatr1}.
If a space ${X}\hspace{-1pt}$ has a Lusin $\pi\mathsurround=0pt$-base, then it has a countable $\hspace{1pt}\mathsurround=0pt\pi$-base and a countable pseudo-base (both with clopen members); and also $X$ is a Choquet space (but it can be not strong Choquet even in a separable metrizable case). For each ${M}\subseteq\omega^\omega$, there exists a separable metrizable space with a Lusin $\pi\mathsurround=0pt$-base that contains a closed subspace homeomorphic to ${M}.$

\section{Notation and terminology}
\label{section1}

We use terminology from~\cite{top.enc} and~\cite{kun}. A \emph{space} is a topological space. Also we use the following notations:

\begin{nota}\label{not01}%
  The symbol $\coloneq$ means ``equals by definition''\textup{;}
  the symbol ${\colon}{\longleftrightarrow}$ is used to show that an expression on the left side is an abbreviation for expression on the right side\textup{;}

  \begin{itemize}
  \item [\ding{46}\,]
    $\mathsurround=0pt
    \omega$ $\coloneq$ the set of finite ordinals $=$ the set of natural numbers,

    so $0=\varnothing\in\omega$ and ${n}=\{0,\ldots,{n}-1\}$ for all ${n}\in\omega;$
  \item [\ding{46}\,]
    $\mathsurround=0pt
    {s}\,$ is a \textit{sequence}$\quad{\colon}{\longleftrightarrow}\quad
    {s}$ is a function such that $\mathsf{domain}({s})\in\omega$ or $\mathsf{domain}({s})=\omega;$
  %  \item [\ding{46}\,]
  %    $\mathsurround=0pt
  %    \lh({s})\:\coloneq\:\mathsf{domain}({s})\quad$ for a sequence ${s};$
  \item [\ding{46}\,]
    if ${s}$ is sequence, then

    $\mathsurround=0pt
    \lh({s})\:\coloneq\:\mathsf{domain}({s});$
 \item [\ding{46}\,]
   $\mathsurround=0pt
   \langle{s}_0,\ldots,{s}_{{n}-1}\rangle$ $\coloneq$
   the sequence ${s}$ such that  $\lh({s})={n}\in\omega$ and  ${s}(i)={s}_{i}$ for all ${i}\in{n};$
 \item [\ding{46}\,]
   $\mathsurround=0pt
   \langle\rangle$ $\coloneq$ the sequence of length {0};
 \item [\ding{46}\,]
   if ${s}=\langle{s}_0,\ldots,{s}_{{n}-1}\rangle,$ then

   $\mathsurround=0pt
   {s}\hspace{0.5pt}^{\frown}{x}
   \:\coloneq\:\langle{s}_0,\ldots,{s}_{{n}-1},{x}\rangle;$
  \item [\ding{46}\,]
    $\mathsurround=0pt
    {f}{\upharpoonright}{A}$ $\coloneq$ the restriction of function ${f}$ to ${A};$
  \item [\ding{46}\,]
    $\mathsurround=0pt
    {A}\subset{B}
    \quad{\colon}{\longleftrightarrow}\quad
    {A}\subseteq{B}\enskip\mathsf{and}\enskip{A}\neq{B};$
  \item [\ding{46}\,]
    if ${s}$ and ${t}$ are sequences, then

    $\mathsurround=0pt
    {s}\sqsubseteq{t}
    \quad{\colon}{\longleftrightarrow}\quad
    {s}={t}{\upharpoonright}\lh({s})\,,$

    $\mathsurround=0pt
    {s}\sqsubset{t}
    \quad{\colon}{\longleftrightarrow}\quad
    {s}\sqsubseteq{t}\enskip\mathsf{and}\enskip{s}\neq{t}$

    (actually, ${s}\sqsubseteq{t}\leftrightarrow{s}\subseteq{t}\;$
    and $\;{s}\sqsubset{t}\leftrightarrow{s}\subset{t}$);
  \item [\ding{46}\,]
    $\mathsurround=0pt
    {}^{B}\!{A}$ $\coloneq$ the set of functions from ${B}$ to ${A};$

    in particular, ${}^{0}\hspace{-1pt}{A}=\big\{\langle\rangle\big\};$
  \item [\ding{46}\,]
    $\mathsurround=0pt
    {}^{{<}\omega}\hspace{-1pt}{A}
    \:\coloneq\:\bigcup_{{n}\in\omega}{}^{{n}}\hspace{-1pt}{A}\,=\,$
    the set of finite sequences in ${A};$
  \item [\ding{46}\,]
    $\mathsurround=0pt
    {A}\,$ is a \textit{singleton}$\quad{\colon}{\longleftrightarrow}\quad
    {A}=\{{x}\}$ for some ${x}\,;$
  \item [\ding{46}\,]
    $\mathsurround=0pt
    \nbhds({p},{X})$ $\coloneq$
    the set of (not necessarily open) neighbourhoods of point ${p}$ in space~${X};$
  %  \item [\ding{46}\,]
  %    $\mathsurround=0pt
  %    \gamma\,$ is a $\pi\mathsurround=0pt$-\textit{net} for a space ${X}
  %    \quad{\colon}{\longleftrightarrow}\quad$for each nonempty open ${U}\subseteq{X}$ there is nonempty ${G}\in\gamma$ such that ${G}\subseteq{U};$
  %  \item [\ding{46}\,]
  %    $\mathsurround=0pt
  %    \gamma\,$ is a $\pi\mathsurround=0pt$-\textit{net} for a space ${X}
  %    \quad{\colon}{\longleftrightarrow}\quad\varnothing\nin\gamma$ and for each nonempty open ${U}\subseteq{X}$ there is ${G}\in\gamma$ such that ${G}\subseteq{U};$
  %  \item [\ding{46}\,]
  %    $\mathsurround=0pt
  %    \gamma\,$ is a $\pi\mathsurround=0pt$-\textit{net} for a space ${X}
  %    \quad{\colon}{\longleftrightarrow}\quad\gamma\not\ni\varnothing$ and for each nonempty open ${U}\subseteq{X}$ there is ${G}\in\gamma$ such that ${G}\subseteq{U};$
  %  \item [\ding{46}\,]
  %    $\mathsurround=0pt
  %    \gamma\,$ is a $\pi\mathsurround=0pt$-\textit{net} for a space ${X}
  %    \quad{\colon}{\longleftrightarrow}\quad$for each nonempty open ${U}\subseteq{X}$ there is ${G}\in\gamma$ such that ${G}\subseteq{U}$ and $\varnothing\nin\gamma;$
  \item [\ding{46}\,]
    $\mathsurround=0pt
    \gamma\,$ is a $\pi\mathsurround=0pt$-\textit{net} for a space ${X}
    \quad{\colon}{\longleftrightarrow}\quad$for each nonempty open ${U}\subseteq{X},$ there is ${G}\in\gamma$ such that ${G}\subseteq{U},$ and all elements of $\gamma$ are nonempty;
  \item [\ding{46}\,]
    $\mathsurround=0pt
    \gamma\,$ is a $\pi\mathsurround=0pt$-\textit{base} for a space ${X}
    \quad{\colon}{\longleftrightarrow}\quad\gamma$ is a $\pi\mathsurround=0pt$-net for ${X}$ and all elements of $\gamma$ are open.
  \end{itemize}
\end{nota}

Recall that~\cite{kech} a \textit{Lusin scheme} on a set $X$ is a family ${L}=\langle{L}_{a}\rangle_{{a}\in{}^{{<}\omega\hspace{-1pt}}\omega}$ of subsets of ${X}$ such that

\begin{itemize}
\item[(L0)]
  $\mathsurround=0pt
  {L}_{a}\supseteq {L}_{{a}\hspace{0.5pt}^{\frown}{n}}\quad$ for all ${a}\in{}^{<\hspace{0.2pt}\omega}\hspace{-1pt}\omega,$
  ${n}\in\omega\enskip$  and
\item[(L1)]
  $\mathsurround=0pt
  {L}_{{a}\hspace{0.5pt}^{\frown}{n}}\cap
  {L}_{{a}\hspace{0.5pt}^{\frown}{m}}=\varnothing\quad$ for all ${a}\in{}^{<\hspace{0.2pt}\omega}\hspace{-1pt}\omega$ and ${n}\neq{m}\in\omega.$
\end{itemize}

\begin{deff}\label{def.luz.p.base}\mbox{ }

  \begin{itemize}
  %  \item [\ding{46}\,]
  %    A \textbf{Lusin scheme} on a set $X$ is a family ${L}=\langle{L}_{a}\rangle_{{a}\in{}^{{<}\omega\hspace{-1pt}}\omega}$ of subsets of ${X}$ such that
  %    \begin{itemize}
  %    \item[(L0)]
  %      $\mathsurround=0pt
  %      {L}_{a}\supseteq {L}_{{a}\hspace{0.5pt}^{\frown}{n}}\quad$ for all ${a}\in{}^{<\hspace{0.2pt}\omega}\hspace{-1pt}\omega,$
  %      ${n}\in\omega\enskip$  and
  %    \item[(L1)]
  %      $\mathsurround=0pt
  %      {L}_{{a}\hspace{0.5pt}^{\frown}{n}}\cap
  %      {L}_{{a}\hspace{0.5pt}^{\frown}{m}}=\varnothing\quad$ for all ${a}\in{}^{<\hspace{0.2pt}\omega}\hspace{-1pt}\omega$ and ${n}\neq{m}\in\omega.$
  %    \end{itemize}
  \item [\ding{46}\,]
    A Lusin scheme ${L}$ on a set $X$ is \textbf{strict} \ iff
    \begin{itemize}
    \item[(L2)]
      $\mathsurround=0pt
      {L}_{\scriptscriptstyle\langle\rangle}={X},$
    \item[(L3)]
      $\mathsurround=0pt
      {L}_{a}=\bigcup_{n\in\omega} {L}_{{a}\hspace{0.5pt}^{\frown}{n}}\quad$ for all ${a}\in{}^{{<}\omega\hspace{-1pt}}\omega,\enskip$ and
    \item[(L4)]
      $\mathsurround=0pt
      \bigcap_{n\in\omega}{L}_{{p}{\upharpoonright}{n}}$ is a singleton for all ${p}\in{}^{\omega\hspace{-1pt}}\omega.$
    \end{itemize}
  \item [\ding{46}\,]
    A Lusin scheme ${L}$ on a space ${X}$ is \textbf{open} \ iff
    \begin{itemize}
    \item[(L5)]
      each ${L}_{a}$ is an open subset of ${X}.$
    \end{itemize}
  \item [\ding{46}\,]
    A \textbf{Lusin $\hspace{1pt}\mathsurround=0pt\pi$-base} for a space $X$ is an open strict Lusin scheme ${L}$ on $X$ such that%with the following property:
    \begin{itemize}
    \item[(L6)]
      $\mathsurround=0pt
      \forall{p}\in{X}\enskip\forall{U}\in\nbhds({p},{X})$

      $\mathsurround=0pt
      \exists {a}\in{}^{<\hspace{0.2pt}\omega}\hspace{-1pt}\omega\enskip
      \exists{k}\in\omega\ $ %      such that
      \begin{itemize}
      \item[\ding{226}\,]
        $\mathsurround=0pt
        {L}_{a}\ni{p}\enskip$ and
      \item[\ding{226}\,]
        $\mathsurround=0pt
        \bigcup_{{j}\geqslant{k}}{L}_{{a}\hspace{0.5pt}^{\frown}{j}}\subseteq {U}.$
      \end{itemize}
    \end{itemize}
  \end{itemize}
\end{deff}

\begin{rema}\label{rem.st.luz.sch}
  Suppose that ${L}$ is a strict Lusin scheme. Then, for all ${a},{b}\in{}^{{<}\omega\hspace{-1pt}}\omega,$

  \begin{itemize}
  \item [(1)]
    $\mathsurround=0pt
    {L}_{{a}}\neq\varnothing;$
  \item [(2)]
    $\mathsurround=0pt
    {L}_{{a}}\supset{L}_{{b}}\ \longleftrightarrow\ {a}\sqsubset{b};$
  \item [(3)]
    $\mathsurround=0pt
    {L}_{{a}}\cap{L}_{{b}}\neq\varnothing\ \longleftrightarrow\
    ({a}\sqsubseteq{b}\textsf{ or }{a}\sqsupseteq{b}).$\hfill$\qed$%
  \end{itemize}
\end{rema}

  %Recall a family $\gamma$ of nonempty open sets in a space ${X}$ is a $\pi\mathsurround=0pt$-base for ${X}$ iff for each nonempty open ${U}\subseteq{X}$ there is some ${V}\in\gamma$ such that ${V}\subseteq{U}.$
  %Recall a family $\gamma$ is a $\pi\mathsurround=0pt$\textit{-base} for a space ${X}$ iff $\gamma$ is a $\pi\mathsurround=0pt$-net for ${X}$ and all elements of $\gamma$ are open.
  %
  %Recall a family $\gamma$ is a $\pi\mathsurround=0pt$\textit{-base} for a space ${X}$ iff $\gamma$ is a $\pi\mathsurround=0pt$-net for ${X}$ and every set in $\gamma$ is open.
  %Recall that a $\pi\mathsurround=0pt$\textit{-base} is a $\pi\mathsurround=0pt$-net with open members.

  %Recall $\gamma$ is a $\pi\mathsurround=0pt$\textit{-base} for a space ${X}$ iff $\gamma$ is a $\pi\mathsurround=0pt$-net for ${X}$ and all elements of $\gamma$ are open.

\begin{rema}\label{pi.base}
  Suppose that ${L}$ is a Lusin $\pi\mathsurround=0pt$-base for a space ${X}.$ Then the family $\{{L}_{{a}}:{a}\in{}^{{<}\omega\hspace{-1pt}}\omega\}$ is a $\pi\mathsurround=0pt$-base for ${X}.$\hfill$\qed$%
  %(that is, each nonempty open ${U}\subseteq{X}$ contains some nonempty ${L}_{{a}}$ as a subset.)\hfill$\qed$
\end{rema}

\newpage
\begin{nota}\label{not.S.N}\mbox{\ }

  \begin{itemize}
  %  \item [\ding{46}\,]
  %     The \textbf{standard Lusin scheme}, which we denote by $\mathbf{S},$ is the Lusin scheme $\langle\mathbf{S}_{a}\rangle_{{a}\in{}^{{<}\omega\hspace{-1pt}}\omega}$ such that
  %     $\mathbf{S}_{{a}}=\{{p}\in{}^{\omega\hspace{-1pt}}\omega:{a}\sqsubseteq{p}\}\quad$
  %     for all ${a}\in{}^{{<}\omega\hspace{-1pt}}\omega;$
  \item [\ding{46}\,]
     $\mathsurround=0pt
     \mathbf{S}$ $\coloneq$ the \textbf{standard Lusin scheme} $\coloneq$ the Lusin scheme $\langle\mathbf{S}_{a}\rangle_{{a}\in{}^{{<}\omega\hspace{-1pt}}\omega}$ such that

     $\mathsurround=0pt
     \mathbf{S}_{{a}}=\{{p}\in{}^{\omega\hspace{-1pt}}\omega:{a}\sqsubseteq{p}\}\quad$
     for all ${a}\in{}^{{<}\omega\hspace{-1pt}}\omega.$
  \item [\ding{46}\,]
     $\mathsurround=0pt
     {\omega^{\omega}}$ $\coloneq$ the \emph{Baire space} $\coloneq$ the space $\langle{}^{\omega\hspace{-1pt}}\omega,\tau_{\mathsf{prod}}\rangle,$ where $\tau_{\mathsf{prod}}$ is the Tychonoff product topology with $\omega$ carrying %endowed with
     the discrete topology.
  \end{itemize}
\end{nota}

\begin{rema}\label{rem.baire.space}
  \begin{itemize}
  \item [(1)]
    The family $\{\mathbf{S}_{{a}}:{a}\in{}^{{<}\omega\hspace{-1pt}}\omega\}$ is a base for the Baire space ${\omega^{\omega}}.$
  \item [(2)]
    The standard Lusin scheme $\mathbf{S}$ is a Lusin $\pi\mathsurround=0pt$-base for ${\omega^{\omega}}.$\hfill$\qed$%
  \end{itemize}
\end{rema}

\begin{nota}
  We denote by $\vartriangleleft$ the strict lexicographic order on ${}^{\omega\hspace{-1pt}}\omega;$ that is,
  for all ${p},{q}\in{}^{\omega\hspace{-1pt}}\omega,$

  \begin{itemize}
  \item[\ding{46}\,]
    $\mathsurround=0pt
    {p}\vartriangleleft{q}
    \quad{\colon}{\longleftrightarrow}\quad
    \exists{n},{i},{j}\in\omega$ such that
    \begin{itemize}
    \item[\ding{226}\,]
      $\mathsurround=0pt
      {p}{\upharpoonright}{n}={q}{\upharpoonright}{n},$
    \item[\ding{226}\,]
      $\mathsurround=0pt
      ({p}{\upharpoonright}{n})\hspace{0.5pt}^{\frown}{i}\sqsubseteq{p},$
    \item[\ding{226}\,]
      $\mathsurround=0pt
      ({q}{\upharpoonright}{n})\hspace{0.5pt}^{\frown}{j}\sqsubseteq{q},\quad$ and
    \item[\ding{226}\,]
      $\mathsurround=0pt
      {i}<{j}.$
    \end{itemize}
  %  \item [\ding{46}\,]
  %    $\mathsurround=0pt
  %    {p}{\downfilledspoon}\:\coloneq\:
  %    \{{q}\in{}^{\omega\hspace{-1pt}}\omega:{p}\vartriangleleft{q}\,\textsf{ or }\,{p}={q}\}.$
  \item [\ding{46}\,]
    $\mathsurround=0pt
    {p}{\downfilledspoon}\:\coloneq\:
    \{{p}\}\cup\{{q}\in{}^{\omega\hspace{-1pt}}\omega:{p}\vartriangleleft{q}\}.$
  \item [\ding{46}\,]
    $\mathsurround=0pt
    {A}{\uparrow}\:\coloneq\:
    \bigcup\{{p}{\downfilledspoon}:{p}\in{A}\}.$
  \end{itemize}
\end{nota}

\begin{rema}\label{rem.<}
  Suppose that ${a}\in{}^{{<}\omega\hspace{-1pt}}\omega$ and ${n},{m}\in\omega.$
  Then

  \begin{itemize}
  \item[\ding{226}\,]
    $\mathsurround=0pt
    \mathbf{S}_{{a}\hspace{0.5pt}^{\frown}{n}}\cap
    (\mathbf{S}_{{a}\hspace{0.5pt}^{\frown}{m}}{\uparrow})=\varnothing\quad$
    whenever ${n}<{m}.$\hfill$\qed$%
  \end{itemize}
\end{rema}

\section{Construction of a space $\langle{}^{\omega\hspace{-1pt}}\omega,\tau_{\scriptscriptstyle\mathcal{F}}\rangle$}

In this section, for each family $\mathcal{F}=\langle\mathcal{F}_{{p}}\rangle_{{p}\in{}^{\omega\hspace{-1pt}}\omega}$ of free filters on a $\omega,$ we construct a topology~$\tau_{\scriptscriptstyle\mathcal{F}}$ on~${}^{\omega\hspace{-1pt}}\omega$ such that the space
$\langle{}^{\omega\hspace{-1pt}}\omega,\tau_{\scriptscriptstyle\mathcal{F}}\rangle$ has a Lusin $\pi\mathsurround=0pt$-base, see Example~\ref{example}. In the next section we build a family $\mathcal{F}$ in such a way that the square of $\langle{}^{\omega\hspace{-1pt}}\omega,\tau_{\scriptscriptstyle\mathcal{F}}\rangle$ has no Lusin $\pi\mathsurround=0pt$-base.

First we reformulate condition (L6) of the definition of a Lusin $\pi\mathsurround=0pt$-base, see Definition~\ref{def.luz.p.base} and Remark~\ref{rem.L6'}, by using the notion of a \textit{shoot}:

\begin{nota}
  Suppose that ${L}$ is a Lusin scheme, ${a}\in{}^{{<}\omega\hspace{-1pt}}\omega,$ and ${k}\in\omega.$  Then

  \begin{itemize}
  \item [\ding{46}\,]
    $\mathsurround=0pt%\displaystyle
    \widetilde{{L}}^{{k}}_{{a}}\:\coloneq\:
    \bigcup_{{j}\geqslant{k}}{L}_{{a}\hspace{0.5pt}^{\frown}{j}};$
  \item [\ding{46}\,]
    $\mathsurround=0pt
    \shoot_{{L}}({a})\:\coloneq\:
    \big\{\widetilde{{L}}^{{k}}_{{a}}:{k}\in\omega\big\};$
  \item [\ding{46}\,]
    $\mathsurround=0pt
    \gamma\inn{U}
    \quad{\colon}{\longleftrightarrow}\quad
    \exists{G}\in\gamma\:[{G}\subseteq{U}].$
  \end{itemize}
\end{nota}

\begin{rema}\label{rem.L6'}
  The clause (L6) of the definition of a Lusin $\pi\mathsurround=0pt$-base is equivalent to the following:
  %Suppose that ${L}$ is an open strict Lusin scheme on a space $X.$ Then the following are equivalent:

  \begin{itemize}
  \item[(L6')]
    $\mathsurround=0pt
    \forall{p}\in{X}\enskip\forall{U}\in\nbhds({p},{X})$

    $\mathsurround=0pt
    \exists {a}\in{}^{<\hspace{0.2pt}\omega}\hspace{-1pt}\omega\ $
    such that
    \begin{itemize}
    \item[\ding{226}\,]
    $\mathsurround=0pt
    {L}_{a}\ni{p}\enskip$ and
    \item[\ding{226}\,]
    $\mathsurround=0pt
    \shoot_{{L}}({a})\inn {U}.$\hfill$\qed$%
    \end{itemize}
  \end{itemize}
\end{rema}

It is convenient to call $\widetilde{{L}}^{{k}}_{{a}}$ a ${k}\mathsurround=0pt$-\textit{tail} of ${L}_{{a}},$ so that the $\shoot_{{L}}({a})$ is the set of tails of ${L}_{{a}}$ and $\shoot_{{L}}({a})\inn{U}$ iff some tail of ${L}_{{a}}$ is contained in ${U}.$ Note also that the family $\shoot_{{L}}({a})$ is closed under finite intersections. We will use the following simple properties of shoots:

\begin{rema}\label{rem.inn}
  Suppose that ${L}$ is a strict Lusin scheme and ${a}\in{}^{{<}\omega\hspace{-1pt}}\omega.$
  Then

  \begin{itemize}
  \item [(1)]
    $\mathsurround=0pt
    \shoot_{{L}}({a})\inn{L}_{{a}}.$
  \item [(2)]
    If $\shoot_{{L}}({a})\inn{A}$ and ${A}\subseteq{B}$,
    then
    $\shoot_{{L}}({a})\inn{B}.$
  \item [(3)]
    If $\shoot_{{L}}({a})\inn{A}$ and $\shoot_{{L}}({a})\inn{B}$,
    then
    $\shoot_{{L}}({a})\inn{A}\cap{B}.$
  \item [(4)]
    If ${A}\cap{B}=\varnothing,$
    then $\shoot_{{L}}({a})\not\inn{A}$ or $\shoot_{{L}}({a})\not\inn{B}.$\hfill$\qed$%
  \end{itemize}
\end{rema}

Next we define sets $\widehat{\mathbf{S}}_{{p}}({M},{g}),$ which will form basic neighbourhoods at points ${p}\in{}^{\omega\hspace{-1pt}}\omega$ for topology $\tau_{\scriptscriptstyle\mathcal{F}}$ in Example~\ref{example}.
Recall that $\mathbf{S}_{a}=\{{p}\in{}^{\omega\hspace{-1pt}}\omega:{a}\sqsubseteq{p}\},$ $\widetilde{\mathbf{S}}^{{k}}_{{a}}=  \bigcup_{{j}\geqslant{k}}\mathbf{S}_{{a}\hspace{0.5pt}^{\frown}{j}},$
and $[\omega]^{\omega}$ is the set of infinite subsets of $\omega.$

\begin{nota}
  Suppose that ${p}\in{}^{\omega\hspace{-1pt}}\omega,$ ${M}\in[\omega]^{\omega},$ and ${g}\colon{M}\to\omega.$ Then

  \begin{itemize}
  \item[\ding{46}\,]
    $\mathsurround=0pt\displaystyle
    \widehat{\mathbf{S}}_{{p}}({M},{g})\:\coloneq\:
    \{{p}\}\,\cup\bigcup_{{m}\in{M}}\widetilde{\mathbf{S}}^{{g}({m})}_{{p}{\upharpoonright}{m}}.$
  \end{itemize}
\end{nota}

\begin{rema}\label{rem.lhd}
  Suppose that
  ${p}\in{}^{\omega\hspace{-1pt}}\omega,$ ${M}\in[\omega]^{\omega},$ ${g}\colon{M}\to\omega,$ and ${g}({m})>{p}({m})$ for all ${m}\in{M}.$
  Then:

  \begin{itemize}
  \item[(1)]
    $\mathsurround=0pt
    \widetilde{\mathbf{S}}^{{g}({m})}_{{p}{\upharpoonright}{m}}
    \,\subseteq\;{\mathbf{S}}_{{p}{\upharpoonright}{m}}
    \!\setminus{\mathbf{S}}_{{p}{\upharpoonright}({m}+1)}\quad$ for all ${m}\in{M}.$
  \item[(2)]
    $\mathsurround=0pt
    \widehat{\mathbf{S}}_{{p}}({M},{g})\,\subseteq\,{p}{\downfilledspoon}\,.$\hfill$\qed$%
  \end{itemize}
\end{rema}

Recall that $\gamma$ is a \textit{filter} on $\omega$ iff $\gamma$ is a family of nonempty subsets of $\omega$ such that (1) ${A}\cap{B}\in\gamma$ for all ${A},{B}\in\gamma$ and (2) ${C}\in\gamma$ whenever $\omega\supseteq{C}\supseteq{A}\in\gamma.$
A filter $\gamma$ is \textit{free} iff $\bigcap\{{A}:{A}\in\gamma\}=\varnothing.$ Now we construct a space $\langle{}^{\omega\hspace{-1pt}}\omega,\tau_{\scriptscriptstyle\mathcal{F}}\rangle,$ which has $\mathbf{S}$ as a Lusin $\pi\mathsurround=0pt$-base:

\begin{exam}\label{example}
  Suppose that $\mathcal{F}=\langle\mathcal{F}_{{p}}\rangle_{{p}\in{}^{\omega\hspace{-1pt}}\omega}$ is an indexed family of free filters on $\omega.$ Then

  \begin{itemize}
  \item[\ding{46}\,]
    $\mathsurround=0pt
    \tau_{\scriptscriptstyle\mathcal{F}}\:\coloneq\:$    the topology on ${}^{\omega\hspace{-1pt}}\omega$ generated by the subbase
    $\big\{\,\widehat{\mathbf{S}}_{{p}}({M},{g})\::\:
    {p}\in{}^{\omega\hspace{-1pt}}\omega,\:{M}\in\mathcal{F}_{{p}},\ \mathsf{and}\ {g}\colon{M}\to\omega\big\}.$
  \end{itemize}
  The space $\langle{}^{\omega\hspace{-1pt}}\omega,\tau_{\scriptscriptstyle\mathcal{F}}\rangle$ has the following properties:

  \begin{itemize}
  \item[(F1)]
    For each point ${p}\in{}^{\omega\hspace{-1pt}}\omega,$ the family  $\big\{\,\widehat{\mathbf{S}}_{{p}}({M},{g}):  {M}\in\mathcal{F}_{{p}}\enskip\mathsf{and}\enskip{g}\colon{M}\to\omega\big\}$
    is a neighbourhood base for $\langle{}^{\omega\hspace{-1pt}}\omega,
    \tau_{\scriptscriptstyle\mathcal{F}}\rangle$ at ${p}.$
  %  \item[(1)]
  %    The family  $\{\,\widehat{\mathbf{S}}_{{p}}({M},{g}):  {M}\in\mathcal{F}_{{p}}\enskip\mathsf{and}\enskip{g}\colon{M}\to\omega\}$
  %    is a neighbourhood base for $\langle{}^{\omega\hspace{-1pt}}\omega,
  %    \tau_{\scriptscriptstyle\mathcal{F}}\rangle$ at point ${p}$ for all ${p}\in{}^{\omega\hspace{-1pt}}\omega;$
  \item[(F2)]
    The standard Lusin scheme $\mathbf{S}$ is a Lusin $\pi\mathsurround=0pt$-base for $\langle{}^{\omega\hspace{-1pt}}\omega,\tau_{\scriptscriptstyle\mathcal{F}}\rangle.$\hfill$\qed$%
  \end{itemize}
\end{exam}

It is interesting to note one additional property of the above space, although we do not use it in this paper.
In~\cite{MPatr3} we found sufficient conditions under which the product of spaces that have a Lusin $\pi\mathsurround=0pt$-base also has a Lusin $\pi\mathsurround=0pt$-base; these conditions were formulated by using the following terminology: For a Lusin scheme ${L}$ on a space ${X},$ a point ${p}\in{X}$ and its neighbourhood ${U}\subseteq{X},$ the \textit{span} of ${U}$ at ${p}$ is defined as follows:
$$
\spa_{{L}}({U},{p})\:\coloneq\:
\{\,\lh({a}):{a}\in{}^{{<}\omega\hspace{-1pt}}\omega,\ \ {L}_{a}\ni{p},\ \ \mathsf{and}\ \ \shoot_{{L}}({a})\inn{U}\,\}\:\subseteq\:\omega
$$
(so that clause (L6') of Remark~\ref{rem.L6'} is equivalent to the assertion that every
$\spa_{{L}}({U},{p})$ is nonempty).
If ${L}$ is a Lusin $\pi\mathsurround=0pt$-base for ${X}$ and ${p}\in{X},$ then
the family $\{\spa_{{L}}({U},{p}):{U}\in\nbhds({p},{X})\}$ has the finite intersection property~\cite[Lemma\,3.6]{MPatr3}, so we can define
the $\rt_{{L}}({p},{X})$ to be the filter on $\omega$ generated by this family.
A theorem in~\cite{MPatr3} says that if some spaces have Lusin $\pi\mathsurround=0pt$-bases with ``good'' roots, then the product of these spaces has a Lusin $\pi\mathsurround=0pt$-base.

Now, the space $\langle{}^{\omega\hspace{-1pt}}\omega,
\tau_{\scriptscriptstyle\mathcal{F}}\rangle$ from Example~\ref{example} has the following property:

\begin{itemize}
\item[(F3)]
  $\mathsurround=0pt
  \rt_{\hspace{1pt}\mathbf{S}}\big({p},\langle{}^{\omega\hspace{-1pt}}\omega,
  \tau_{\scriptscriptstyle\mathcal{F}}\rangle\big)=\mathcal{F}_{{p}}$ \quad for all ${p}\in{}^{\omega\hspace{-1pt}}\omega.$
\end{itemize}
Therefore we can build a space with a Lusin $\pi\mathsurround=0pt$-base that has any roots we wish.

\section{Construction of a space $\langle{}^{\omega\hspace{-1pt}}\omega,\tau_{\scriptscriptstyle\mathcal{F}}\rangle$ whose square
\\ has no Lusin $\pi\mathsurround=0pt$-base}

In this section we prove Theorem~\ref{teo}, the main result of the article, which says that there exists a space with a Lusin $\pi\mathsurround=0pt$-base whose square has no Lusin $\pi\mathsurround=0pt$-base.
To build such a space we construct a family $\mathcal{F}$ of free filters on $\omega$ in such a way that the space $\langle{}^{\omega\hspace{-1pt}}\omega,\tau_{\scriptscriptstyle\mathcal{F}}\rangle^{2}$ has no Lusin $\pi\mathsurround=0pt$-base. We construct this family in the proof of Lemma~\ref{lemma.main}, so Theorem~\ref{teo} is a corollary to this lemma. To make the logic of construction more transparent, we state Lemmas~\ref{continuum.igr},~\ref{cand},~\ref{lem.sigma_L},~\ref{lem.M^i.g^i} without proofs and prove them in Section~\ref{sect.proofs.of.lemmas}.

\begin{nota}
  \mbox{ }

  \begin{itemize}
  \item[\ding{46}\,]
    A \textbf{game} is an ordered pair of sequences $\Gamma=\big\langle\langle{c}_{{n}}\rangle_{{n}\in\omega},  \langle{d}_{{n}}\rangle_{{n}\in\omega}\big\rangle$ such that, for all ${n}\in\omega,$
    \begin{itemize}
    \item[\ding{226}\,]
      $\mathsurround=0pt
      {c}_{{n}}=\langle{c}^{0}_{{n}},{c}^{1}_{{n}}\rangle\in{}^{{<}\omega\hspace{-1pt}}\omega
      \times{}^{{<}\omega\hspace{-1pt}}\omega,$
    \item[\ding{226}\,]
      $\mathsurround=0pt
      {d}_{{n}}=\langle{d}^{0}_{{n}},{d}^{1}_{{n}}\rangle\in{}^{{<}\omega\hspace{-1pt}}\omega
      \times{}^{{<}\omega\hspace{-1pt}}\omega,$
    \item[\ding{226}\,]
      $\mathsurround=0pt
      {c}^{0}_{{n}}\sqsubset{d}^{0}_{{n}}\sqsubset{c}^{0}_{{n}+1},\,$ and
      $\,{c}^{1}_{{n}}\sqsubset{d}^{1}_{{n}}\sqsubset{c}^{1}_{{n}+1}.$
    \end{itemize}
  \item[\ding{46}\,]
    The \textbf{result} of a game $\Gamma$ is the pair $\big\langle\rr^{0}(\Gamma),\rr^{1}(\Gamma)\big\rangle
    \in{}^{\omega\hspace{-1pt}}\omega\times{}^{\omega\hspace{-1pt}}\omega$ such that, for all ${n}\in\omega,$
    \begin{itemize}
    \item[\ding{226}\,]
      $\mathsurround=0pt
      {c}^{0}_{{n}}\sqsubseteq\rr^{0}(\Gamma)\,$ and
      ${c}^{1}_{{n}}\sqsubseteq\rr^{1}(\Gamma).$
    \end{itemize}
  \item[\ding{46}\,]
    A \textbf{strategy} is an ordered pair $\sigma=\langle\sigma^{0},\sigma^{1}\rangle$ of functions $\sigma^{0},\sigma^{1}\colon{}^{{<}\omega\hspace{-1pt}}\omega\times{}^{{<}\omega\hspace{-1pt}}\omega\to
    {}^{{<}\omega\hspace{-1pt}}\omega$ such that for each ${c}=\langle{c}^{0},{c}^{1}\rangle\in{}^{{<}\omega\hspace{-1pt}}\omega\times{}^{{<}\omega\hspace{-1pt}}\omega,$
    \begin{itemize}
    \item[\ding{226}\,]
      $\mathsurround=0pt
      {c}^{0}\sqsubset\sigma^{0}({c})$ and
      ${c}^{1}\sqsubset\sigma^{1}({c}).$
    \end{itemize}
  \item[\ding{46}\,]
    A game $\Gamma$ \textbf{follows} a strategy $\sigma$
    iff\; for all ${n}\in\omega,$
    \begin{itemize}
    \item[\ding{226}\,]
      $\mathsurround=0pt
      {d}^{0}_{{n}}=\sigma^{0}({c}_{{n}})$ and ${d}^{1}_{{n}}=\sigma^{1}({c}_{{n}}).$
    \end{itemize}

  \end{itemize}
\end{nota}

\begin{rema}\label{rem.game}
  Suppose that
  $\Gamma=\big\langle\langle{c}_{{n}}\rangle_{{n},{m}\in\omega},  \langle{d}_{{n}}\rangle_{{n}\in\omega}\big\rangle$ is a game.
  Then, for all ${i}\in{2}=\{0,1\}$ and ${n},{m}\in\omega,$

  \begin{itemize}
  \item[(1)]
    $\mathsurround=0pt
    \mathbf{S}_{{c}^{{i}}_{{n}}}\,\supset\:\mathbf{S}_{{d}^{{i}}_{{n}}}
    \,\supset\:\mathbf{S}_{{c}^{{i}}_{{n}+1}}\ni\:\rr^{{i}}(\Gamma);$
  \item[(2)]
    $\mathsurround=0pt
    {c}^{i}_{{n}}=\rr^{{i}}(\Gamma){\upharpoonright}\lh({c}^{i}_{{n}});$
  \item[(3)]
    $\mathsurround=0pt
    {d}^{i}_{{n}}=\rr^{{i}}(\Gamma){\upharpoonright}\lh({d}^{i}_{{n}});$
  \item[(4)]
    $\mathsurround=0pt
    \lh({c}^{i}_{{n}+1})-1\geqslant\lh({d}^{i}_{{n}});$
  \item[(5)]
    $\mathsurround=0pt
    {c}^{{i}}_{{n}}\sqsubseteq{c}^{{i}}_{{m}}
    \ \longleftrightarrow\ {n}\leqslant{m};$
  \item[(6)]
    $\mathsurround=0pt
    {d}^{{i}}_{{n}}\sqsubseteq{d}^{{i}}_{{m}}
    \ \longleftrightarrow\ {n}\leqslant{m}.$\hfill$\qed$%
  \end{itemize}

\end{rema}

\begin{lemm}\label{continuum.igr}
  For each strategy $\sigma,$ there exists a family $\langle\Gamma_{{z}}\rangle_{{z}\in{}^{\omega}{2}}$ of games that follow $\sigma$ such that $\rr^{i}(\Gamma_{{z}})\neq\rr^{{j}}(\Gamma_{{y}})$
  for all $\langle{z},{i}\rangle\neq\langle{y},{j}\rangle\in{}^{\omega}{2}\times{2}.$
\end{lemm}

Recall that a family $\gamma$ is a $\pi\mathsurround=0pt$-net for a space ${X}$ iff $\varnothing\nin\gamma$ and each nonempty open ${U}\subseteq{X}$ contains some ${G}\in\gamma$ as a subset. Recall also that ${\omega^{\omega}}$ is the Baire space $\langle{}^{\omega\hspace{-1pt}}\omega,\tau_{\mathsf{prod}}\rangle.$

  %\begin{nota}
  %  A \textbf{candidate} is a strict Lusin scheme ${L}$ on the set
  %  ${}^{\omega\hspace{-1pt}}\omega\times{}^{\omega\hspace{-1pt}}\omega$
  %  such that the family $\big\{\mathsf{interior}({L}_{{b}},{\omega^{\omega}}\times{\omega^{\omega}}):
  %  {b}\in{}^{{<}\omega\hspace{-1pt}}\omega\big\}$ is a $\pi\mathsurround=0pt$-base for the space ${\omega^{\omega}}\times{\omega^{\omega}}.$
  %\end{nota}

\begin{nota}
  A \textbf{candidate} is a strict Lusin scheme ${L}$ on the set
  ${}^{\omega\hspace{-1pt}}\omega\times{}^{\omega\hspace{-1pt}}\omega$
  such that $\{{L}_{{b}}:{b}\in{}^{{<}\omega\hspace{-1pt}}\omega\}$ is a $\pi\mathsurround=0pt$-net for the space ${\omega^{\omega}}\times{\omega^{\omega}}$
  and each ${L}_{{b}}$ has nonempty interior in ${\omega^{\omega}}\times{\omega^{\omega}}.$
\end{nota}

\begin{lemm}\label{cand}
  Suppose that the standard Lusin scheme $\mathbf{S}$ is a Lusin $\pi\mathsurround=0pt$-base for a space $\langle{}^{\omega\hspace{-1pt}}\omega,\tau\rangle$
  and $\,{L}\!$ is a Lusin $\pi\mathsurround=0pt$-base for $\langle{}^{\omega\hspace{-1pt}}\omega,\tau\rangle\times\langle{}^{\omega\hspace{-1pt}}\omega,\tau\rangle.$
  Then ${L}$ is a candidate.
\end{lemm}

\begin{lemm}\label{lem.sigma_L}
  For each candidate ${L},$
  there exists a strategy $\sigma$
  that possesses the following property:

  \begin{itemize}
  \item[\textup{(\ding{72})}]
    For each ${c}=\langle{c}^{0},{c}^{1}\rangle\in{}^{{<}\omega\hspace{-1pt}}\omega\times{}^{{<}\omega\hspace{-1pt}}\omega,$

    there is ${b}\in{}^{{<}\omega\hspace{-1pt}}\omega$ such that
    \begin{itemize}
    \item[\ding{226}\,]
      $\mathsurround=0pt
      \mathbf{S}_{{c}^{0}}\times\mathbf{S}_{{c}^{1}}\:\supseteq\:
      {L}_{{b}}\:\supseteq\:
      \mathbf{S}_{\sigma^{0}({c})}\times
      \mathbf{S}_{\sigma^{1}({c})}\quad$ and
    \item[\ding{226}\,]
      $\mathsurround=0pt
      \shoot_{{L}}({b}{\upharpoonright}{k})\:\not\inn\:
      \big(\mathbf{S}_{\sigma^{0}({c})}{\uparrow}
      \times\mathbf{S}_{\sigma^{1}({c})}\big)\cup
      \big(\mathbf{S}_{\sigma^{0}({c})}
      \times(\mathbf{S}_{\sigma^{1}({c})}{\uparrow})\big)\quad$
      for all ${k}<\lh({b}).$
    \end{itemize}
  \end{itemize}
\end{lemm}

\begin{nota}\label{not.sigma_L}
  Suppose that ${L}$ is a candidate.
  Then we denote by $\sigma_{\hspace{-1pt}\scriptscriptstyle{L}}$ some (fixed) strategy $\sigma$ such that property \textup{(\ding{72})} holds.
\end{nota}

\begin{lemm}\label{lem.M^i.g^i}
  For each game $\Gamma,$  there are sets ${M}^{0},{M}^{1}\in[\omega]^{\omega}$ and functions ${g}^{0}\colon{M}^{0}\to\omega,$ ${g}^{1}\colon{M}^{1}\to\omega$ that possess the following property:

\begin{itemize}
  \item[\textup{(\ding{117})}]
    For each candidate ${L}$
    and each ${a}\in{}^{{<}\omega\hspace{-1pt}}\omega,$
    \begin{itemize}
    \item[\ding{226}\,]
      if game $\Gamma$ follows strategy $\sigma_{\hspace{-1pt}\scriptscriptstyle{L}}$ and
      $\,{L}_{{a}}\ni\langle\rr^{0}(\Gamma),\rr^{1}(\Gamma)\rangle,$

    \item[\ding{226}\,]
      then
      $\ \shoot_{{L}}({a})\:\not\inn\:
      \widehat{\mathbf{S}}_{\rr^{0}(\Gamma)}({M}^{0},{g}^{0})\times
      \widehat{\mathbf{S}}_{\rr^{1}(\Gamma)}({M}^{1},{g}^{1}).$
    \end{itemize}
  \end{itemize}
\end{lemm}

\begin{nota}\label{not.M^i.g^i}
  Suppose that $\Gamma$ is a game.
  Then we denote by ${M}^{0}(\Gamma),$ ${M}^{1}(\Gamma),$ ${g}^{0}(\Gamma),$ ${g}^{1}(\Gamma)$ some (fixed) sets ${M}^{0},{M}^{1}$ and functions ${g}^{0},{g}^{1}$ such that property \textup{(\ding{117})} holds.
\end{nota}

\begin{lemm}\label{lemma.main}
  There exists an indexed family  $\mathcal{F}=\langle\mathcal{F}_{{p}}\rangle_{{p}\in{}^{\omega\hspace{-1pt}}\omega}$ of free filters on $\omega$ such that the space $\langle{}^{\omega\hspace{-1pt}}\omega,\tau_{\scriptscriptstyle\mathcal{F}}\rangle\times
  \langle{}^{\omega\hspace{-1pt}}\omega,\tau_{\scriptscriptstyle\mathcal{F}}\rangle$ \textup{(see Example~\ref{example})} has no Lusin $\pi\mathsurround=0pt$-base.
\end{lemm}

\begin{proof}
  We may assume that $\{\sigma({x}):{x}\in{}^{\omega}{2}\}$ is the set of all strategies,  because the cardinality of this set equals the cardinality of ${}^{\omega}{2}.$
  Using Lemma~\ref{continuum.igr}, we can build (by transfinite recursion on ${}^{\omega}{2},$ well-ordered in the type of its cardinality) an indexed  family $\langle\Gamma_{{x}}\rangle_{{x}\in{}^{\omega}{2}}$ of games such that

  \begin{itemize}
  \item[(A1)]
    game $\Gamma_{{x}}$ follows strategy $\sigma({x})\quad$ for all ${{x}\in{}^{\omega}{2}}\quad$ and
  \item[(A2)]
    $\mathsurround=0pt
    \rr^{{i}}(\Gamma_{{x}})\neq\rr^{{j}}(\Gamma_{{y}})\quad$ for all $\langle{x},{i}\rangle\neq\langle{y},{j}\rangle\in{}^{\omega}{2}\times{2}.$
  \end{itemize}
  Next, using Notation~\ref{not.M^i.g^i}, for every ${{x}\in{}^{\omega}{2}},$ we set
  $$
  {r}^{0}_{{x}}\:\coloneq\:\rr^{0}(\Gamma_{{x}})\in{}^{\omega\hspace{-1pt}}\omega,\quad  {M}^{0}_{{x}}\:\coloneq\:{M}^{0}(\Gamma_{{x}})\in[\omega]^{\omega},\quad
  {g}^{0}_{{x}}\:\coloneq\:{g}^{0}(\Gamma_{{x}})\in{}^{{M}^{0}_{{x}}}\omega,
  $$
  $$
  {r}^{1}_{{x}}\:\coloneq\:\rr^{1}(\Gamma_{{x}})\in{}^{\omega\hspace{-1pt}}\omega,\quad
  {M}^{1}_{{x}}\:\coloneq\:{M}^{0}(\Gamma_{{x}})\in[\omega]^{\omega},\quad
  {g}^{1}_{{x}}\:\coloneq\:{g}^{1}(\Gamma_{{x}})\in{}^{{M}^{1}_{{x}}}\omega.
  $$
  Now we define an indexed family $\mathcal{F}=\langle\mathcal{F}_{{p}}\rangle_{{p}\in{}^{\omega\hspace{-1pt}}\omega}$
  of free filters on $\omega.$ Since all ${r}^{{i}}_{{x}}$ are different by (A2), we may,
  for each pair $\langle{x},{i}\rangle\in{}^{\omega}{2}\times{2},$
  take $\mathcal{F}_{{r}^{i}_{{x}}}$ to be an arbitrary free filter on $\omega$ such that $\mathcal{F}_{{r}^{i}_{{x}}}\ni{M}^{{i}}_{{x}}.$ Otherwise, if ${p}$ differs from all ${r}^{i}_{{x}},$ set $\mathcal{F}_{{p}}$ to be an arbitrary free filter on $\omega.$
  %  Using (F2), which says that all ${r}^{{i}}_{{x}}$ are different, we can build an indexed family $\mathcal{F}=\langle\mathcal{F}_{{p}}\rangle_{{p}\in{}^{\omega\hspace{-1pt}}\omega}$ of free filters on $\omega$ such that $\mathcal{F}_{{r}^{0}_{{x}}}\ni{M}^{0}_{{x}}$
  %  and $\mathcal{F}_{{r}^{1}_{{x}}}\ni{M}^{1}_{{x}}$ for all ${x}\in{}^{\omega}{2}.$

  We must prove that the space $\langle{}^{\omega\hspace{-1pt}}\omega,\tau_{\scriptscriptstyle\mathcal{F}}\rangle\times
  \langle{}^{\omega\hspace{-1pt}}\omega,\tau_{\scriptscriptstyle\mathcal{F}}\rangle$
  has no Lusin $\pi\mathsurround=0pt$-base.
  Suppose, on the contrary, that ${L}$ is a Lusin $\pi\mathsurround=0pt$-base for
  $\langle{}^{\omega\hspace{-1pt}}\omega,\tau_{\scriptscriptstyle\mathcal{F}}\rangle\times
  \langle{}^{\omega\hspace{-1pt}}\omega,\tau_{\scriptscriptstyle\mathcal{F}}\rangle.$
  Since the standard Lusin scheme $\mathbf{S}$ is a Lusin $\pi\mathsurround=0pt$-base for $\langle{}^{\omega\hspace{-1pt}}\omega,\tau_{\scriptscriptstyle\mathcal{F}}\rangle$
  (see (F2) of Example~\ref{example}),
  it follows by Lemma~\ref{cand} that ${L}$ is candidate, so there is ${x}\in{}^{\omega}{2}$ such that $\sigma({x})=\sigma_{\hspace{-1pt}\scriptscriptstyle{L}}$
  (see Notation~\ref{not.sigma_L}).
  Since ${M}^{0}_{{x}}\in\mathcal{F}_{{r}^{0}_{{x}}}$ and
  ${g}^{0}_{{x}}\colon{M}^{0}_{{x}}\to\omega,$
  it follows by (F1) of Example~\ref{example} that
  $\widehat{\mathbf{S}}_{{r}^{0}_{{x}}}({M}^{0}_{{x}},{g}^{0}_{{x}})$
  is a neighbourhood of point ${r}^{0}_{{x}}$ in the space $\langle{}^{\omega\hspace{-1pt}}\omega,\tau_{\scriptscriptstyle\mathcal{F}}\rangle;$
  similarly, $\widehat{\mathbf{S}}_{{r}^{1}_{{x}}}({M}^{1}_{{x}},{g}^{1}_{{x}})$
  is a neighbourhood of ${r}^{1}_{{x}}$ in the same space.  Then
  $$
  \widehat{\mathbf{S}}_{{r}^{0}_{{x}}}({M}^{0}_{{x}},{g}^{0}_{{x}})\times
  \widehat{\mathbf{S}}_{{r}^{1}_{{x}}}({M}^{1}_{{x}},{g}^{1}_{{x}})
  \:\in\:
  \nbhds\big(\langle{r}^{0}_{{x}},{r}^{1}_{{x}}\rangle,
  \langle{}^{\omega\hspace{-1pt}}\omega,\tau_{\scriptscriptstyle\mathcal{F}}\rangle\times
  \langle{}^{\omega\hspace{-1pt}}\omega,\tau_{\scriptscriptstyle\mathcal{F}}\rangle\big),
  $$
  so, by Remark~\ref{rem.L6'}, there is ${a}\in{}^{{<}\omega\hspace{-1pt}}\omega$ such that  $$
  {L}_{a}\:\ni\:\langle{r}^{0}_{{x}},{r}^{1}_{{x}}\rangle\quad\text{and}\quad
  \shoot_{{L}}({a})
  \:\inn\;
  \widehat{\mathbf{S}}_{{r}^{0}_{{x}}}({M}^{0}_{{x}},{g}^{0}_{{x}})\times
  \widehat{\mathbf{S}}_{{r}^{1}_{{x}}}({M}^{1}_{{x}},{g}^{1}_{{x}}).
  $$
  On the other hand, game $\Gamma_{{x}}$ follows strategy $\sigma({x})=\sigma_{\hspace{-1pt}\scriptscriptstyle{L}}$ by (A1) and
  ${L}_{{a}}\ni\langle{r}^{0}_{{x}},{r}^{1}_{{x}}\rangle=
  \langle\rr^{0}(\Gamma_{{x}}),\rr^{1}(\Gamma_{{x}})\rangle,$ so
  we get a contradiction with
  $$
  \shoot_{{L}}({a})\:\not\inn\:
      \widehat{\mathbf{S}}_{{r}^{0}_{{x}}}({M}^{0}_{{x}},{g}^{0}_{{x}})\times
      \widehat{\mathbf{S}}_{{r}^{1}_{{x}}}({M}^{1}_{{x}},{g}^{1}_{{x}}),
  $$
  which follows from the choice of ${M}^{0}_{{x}},$ ${M}^{1}_{{x}},$ ${g}^{0}_{{x}},$ and  ${g}^{1}_{{x}}.$
\end{proof}

It looks like we did not use Lemmas~\ref{lem.sigma_L},~\ref{lem.M^i.g^i} in the above proof; actually, these lemmas were needed to make Notations~\ref{not.sigma_L},~\ref{not.M^i.g^i} correct. As a corollary to Lemma~\ref{lemma.main} and (F2) of Example~\ref{example}, we get the main result of the article:

\begin{teo}\label{teo}
  There exists a space with a Lusin $\pi\mathsurround=0pt$-base whose square has no Lusin $\pi\mathsurround=0pt$-base.\hfill$\qed$%
\end{teo}

\section{Proofs of Lemmas}\label{sect.proofs.of.lemmas}

\begin{proof}[\textbf{\textup{Proof of Lemma~\ref{continuum.igr}}}]
  This lemma says that for each strategy $\sigma,$ there exists a family $\langle\Gamma_{{z}}\rangle_{{z}\in{}^{\omega}{2}}$ of games that follow $\sigma$ such that $\rr^{i}(\Gamma_{{z}})\neq\rr^{{j}}(\Gamma_{{y}})$
  for all $\langle{z},{i}\rangle\neq\langle{y},{j}\rangle\in{}^{\omega}{2}\times{2}.$

  In this proof we use the following notation: for ${a},{b}\in{}^{{<}\omega\hspace{-1pt}}\omega,$ ${a}\parallel{b}$ means ${a}\nsqsubseteq{b}$ and ${a}\nsqsupseteq{b}.$
  It is not hard to build
  $\check{c}^{i}_{{u}},\check{d}^{i}_{{u}}\in{}^{{<}\omega\hspace{-1pt}}\omega$ for all ${u}\in{}^{{<}\omega}{2}$ and ${i}\in{2}$ such that the following holds:

  \begin{itemize}
  \item[(B1)]
    $\mathsurround=0pt
    \check{d}^{i}_{{u}}=\sigma^{i}(\langle\check{c}^{0}_{{u}},\check{c}^{1}_{{u}}\rangle)\quad$
    for all ${u}\in{}^{{<}\omega}{2}$ and ${i}\in{2};$
  \item[(B2)]
    $\mathsurround=0pt
    \check{d}^{i}_{{u}}\sqsubset\check{c}^{i}_{{u}\hspace{0.5pt}^{\frown}{0}}\,$ and
    $\check{d}^{i}_{{u}}\sqsubset\check{c}^{i}_{{u}\hspace{0.5pt}^{\frown}{1}}\quad$
    for all ${u}\in{}^{{<}\omega}{2}$ and ${i}\in{2};$
  \item[(B3)]
    if ${n}\in\omega,$
    then

    $\mathsurround=0pt
    \check{c}^{i}_{v}\parallel\check{c}^{j}_{w}\quad$
    for all $\langle{v},{i}\rangle\neq\langle{w},{j}\rangle\in{}^{{n}}{2}\times{2}.$
  \end{itemize}
  Since $\sigma$ is a strategy, it follows from (B1)--(B2) that for each ${z}\in{}^{\omega}{2},$ there is a unique game
  $\Gamma=\big\langle\langle{c}_{{n}}\rangle_{{n}\in\omega},  \langle{d}_{{n}}\rangle_{{n}\in\omega}\big\rangle$
  such that
  ${c}_{{n}}=\langle\check{c}^{0}_{{z}{\upharpoonright}{n}},
  \check{c}^{1}_{{z}{\upharpoonright}{n}}\rangle$ and
  ${d}_{{n}}=\langle\check{d}^{0}_{{z}{\upharpoonright}{n}},
  \check{d}^{1}_{{z}{\upharpoonright}{n}}\rangle$
  for all ${n}\in\omega.$ We denote this game by $\Gamma_{{z}}.$
  Each game $\Gamma_{{z}}$ follows strategy $\sigma$ by (B1). Now, suppose that
  $\langle{z},{i}\rangle\neq\langle{y},{j}\rangle\in{}^{\omega}{2}\times{2}.$
  There is ${n}\in\omega$ such that $\langle{z}{\upharpoonright}{n},{i}
  \rangle\neq\langle{y}{\upharpoonright}{n},{j}\rangle,$ so
  $\check{c}^{{i}}_{{z}{\upharpoonright}{n}}\parallel
  \check{c}^{{j}}_{{y}{\upharpoonright}{n}}$ by (B3).
  Then $\rr^{i}(\Gamma_{{z}})\neq\rr^{{j}}(\Gamma_{{y}})$
  because $\rr^{i}(\Gamma_{{z}})\sqsupseteq\check{c}^{{i}}_{{z}{\upharpoonright}{n}}$
  and $\rr^{{j}}(\Gamma_{{y}})\sqsupseteq\check{c}^{{j}}_{{y}{\upharpoonright}{n}}.$
\end{proof}

\begin{proof}[\textbf{\textup{Proof of Lemma~\ref{cand}}}]
  This lemma says that if $\mathbf{S}$ is a Lusin $\pi\mathsurround=0pt$-base for $\langle{}^{\omega\hspace{-1pt}}\omega,\tau\rangle$
  and $\,{L}\!$ is a Lusin $\pi\mathsurround=0pt$-base for $\langle{}^{\omega\hspace{-1pt}}\omega,\tau\rangle\times
  \langle{}^{\omega\hspace{-1pt}}\omega,\tau\rangle,$
  then ${L}$ is a strict Lusin scheme on the set
  ${}^{\omega\hspace{-1pt}}\omega\times{}^{\omega\hspace{-1pt}}\omega,$
  $\{{L}_{{b}}:{b}\in{}^{{<}\omega\hspace{-1pt}}\omega\}$ is a $\pi\mathsurround=0pt$-net for the space ${\omega^{\omega}}\times{\omega^{\omega}},$
  and each ${L}_{{b}}$ has nonempty interior in ${\omega^{\omega}}\times{\omega^{\omega}}.$

  Recall that ${\omega^{\omega}}=\langle{}^{\omega\hspace{-1pt}}\omega,\tau_{\mathsf{prod}}\rangle.$
  Since $\mathbf{S}$ is a Lusin $\pi\mathsurround=0pt$-base for $\langle{}^{\omega\hspace{-1pt}}\omega,\tau\rangle,$
  it follows that each $\mathbf{S}_{{a}}$ is open in $\langle{}^{\omega\hspace{-1pt}}\omega,\tau\rangle,$ so

  \begin{itemize}
  \item[(C)]
    $\mathsurround=0pt
    \tau\supseteq\tau_{\mathsf{prod}}$
  \end{itemize}
  because
  $\{\mathbf{S}_{{a}}:{a}\in{}^{{<}\omega\hspace{-1pt}}\omega\}$ is a base for ${\omega^{\omega}}$
  by (1) Remark~\ref{rem.baire.space}.
  Since ${L}$ is a Lusin $\pi\mathsurround=0pt$-base for $\langle{}^{\omega\hspace{-1pt}}\omega,\tau\rangle\times
  \langle{}^{\omega\hspace{-1pt}}\omega,\tau\rangle,$
  ${L}$ is a strict Lusin scheme on ${}^{\omega\hspace{-1pt}}\omega\times
  {}^{\omega\hspace{-1pt}}\omega.$
  The family $\{{L}_{{b}}:{b}\in{}^{{<}\omega\hspace{-1pt}}\omega\}$ is a $\pi\mathsurround=0pt$-base for $\langle{}^{\omega\hspace{-1pt}}\omega,\tau\rangle\times
  \langle{}^{\omega\hspace{-1pt}}\omega,\tau\rangle$ by Remark~\ref{pi.base}, so $\{{L}_{{b}}:{b}\in{}^{{<}\omega\hspace{-1pt}}\omega\}$ is a $\pi\mathsurround=0pt$-net for ${\omega^{\omega}}\times{\omega^{\omega}}$ by (C).
  % (and by (1) of Remark~\ref{rem.st.luz.sch}).
  Now, for each ${L}_{{b}}$ there are nonempty ${U}^{{0}},{U}^{{1}}\in\tau$ such that
  ${L}_{{b}}\supseteq{U}^{{0}}\times{U}^{{1}}$ because
  ${L}_{{b}}$ is nonempty and open in $\langle{}^{\omega\hspace{-1pt}}\omega,\tau\rangle\times
  \langle{}^{\omega\hspace{-1pt}}\omega,\tau\rangle.$
  Then there are ${a}^{{0}},{a}^{{1}}\in{}^{{<}\omega\hspace{-1pt}}\omega$ such that
  ${U}^{{0}}\supseteq\mathbf{S}_{{a}^{{0}}}$ and
  ${U}^{{1}}\supseteq\mathbf{S}_{{a}^{{1}}}$ because
  $\{\mathbf{S}_{{a}}:{a}\in{}^{{<}\omega\hspace{-1pt}}\omega\}$ is a $\pi\mathsurround=0pt$-base for $\langle{}^{\omega\hspace{-1pt}}\omega,\tau\rangle$ by Remark~\ref{pi.base}. Therefore
  ${L}_{{b}}\supseteq\mathbf{S}_{{a}^{{0}}}\times\mathbf{S}_{{a}^{{1}}},$
  so the interior of ${L}_{{b}}$ in ${\omega^{\omega}}\times{\omega^{\omega}}$ contains
  $\mathbf{S}_{{a}^{{0}}}\times\mathbf{S}_{{a}^{{1}}},$ which in not empty.
\end{proof}

\begin{proof}[\textbf{\textup{Proof of Lemma~\ref{lem.sigma_L}}}]
  This lemma says that for each candidate ${L},$
  there exists a strategy $\sigma$
  that possesses the following property:

  \begin{itemize}
  \item[\textup{(\ding{72})}]
    For each ${c}=\langle{c}^{0},{c}^{1}\rangle\in{}^{{<}\omega\hspace{-1pt}}\omega\times{}^{{<}\omega\hspace{-1pt}}\omega,$

    there is ${b}\in{}^{{<}\omega\hspace{-1pt}}\omega$ such that
    \begin{itemize}
    \item[\ding{226}\,]
      $\mathsurround=0pt
      \mathbf{S}_{{c}^{0}}\times\mathbf{S}_{{c}^{1}}\:\supseteq\:
      {L}_{{b}}\:\supseteq\:
      \mathbf{S}_{\sigma^{0}({c})}\times
      \mathbf{S}_{\sigma^{1}({c})}\quad$ and
    \item[\ding{226}\,]
      $\mathsurround=0pt
      \shoot_{{L}}({b}{\upharpoonright}{k})\:\not\inn\:
      \big(\mathbf{S}_{\sigma^{0}({c})}{\uparrow}
      \times\mathbf{S}_{\sigma^{1}({c})}\big)\cup
      \big(\mathbf{S}_{\sigma^{0}({c})}
      \times(\mathbf{S}_{\sigma^{1}({c})}{\uparrow})\big)\quad$
      for all ${k}<\lh({b}).$
    \end{itemize}
  \end{itemize}

  Suppose that ${L}$ is a candidate; we must find a pair $\sigma=\langle\sigma^{0},\sigma^{1}\rangle$ of
  functions $\sigma^{0},\sigma^{1}\colon{}^{{<}\omega\hspace{-1pt}}\omega\times{}^{{<}\omega\hspace{-1pt}}\omega\to {}^{{<}\omega\hspace{-1pt}}\omega$ such that (\ding{72}) holds.
  Let ${c}=\langle{c}^{0},{c}^{1}\rangle\in{}^{{<}\omega\hspace{-1pt}}\omega\times{}^{{<}\omega\hspace{-1pt}}\omega.$
  The set $\mathbf{S}_{{c}^{0}}\times\mathbf{S}_{{c}^{1}}$ is nonempty and open in the space ${\omega^{\omega}}\times{\omega^{\omega}},$
  so there is ${b}_{{c}}\in{}^{{<}\omega\hspace{-1pt}}\omega$ such that
  ${L}_{{b}_{{c}}}\subseteq\mathbf{S}_{{c}^{0}}\times\mathbf{S}_{{c}^{1}}$
  because $\{{L}_{{b}}:{b}\in{}^{{<}\omega\hspace{-1pt}}\omega\}$ is a $\pi\mathsurround=0pt$-net for ${\omega^{\omega}}\times{\omega^{\omega}}.$
  Then there are
  ${a}_{{c}}^{0},{a}_{{c}}^{1}\in{}^{{<}\omega\hspace{-1pt}}\omega$
  such that
  $\mathbf{S}_{{a}_{{c}}^{0}}\times\mathbf{S}_{{a}_{{c}}^{1}}\subseteq{L}_{{b}_{{c}}}$
  because ${L}_{{b}_{{c}}}$ has nonempty interior in ${\omega^{\omega}}\times{\omega^{\omega}}$ and
  $\{\mathbf{S}_{{a}}:{a}\in{}^{{<}\omega\hspace{-1pt}}\omega\}$ is a base for ${\omega^{\omega}}.$ For each ${n}\in\omega,$ put

  \begin{itemize}
  \item [\ding{46}\,]
    $\mathsurround=0pt
    {R}_{{c},{n}}\:\coloneq\:
    \big(\mathbf{S}_{{a}_{{c}}^{{0}{\frown}}{n}}{\uparrow} \times\mathbf{S}_{{a}_{{c}}^{{1}{\frown}}{n}}\big)\:\cup\: \big(\mathbf{S}_{{a}_{{c}}^{{0}{\frown}}{n}}    \times(\mathbf{S}_{{a}_{{c}}^{{1}{\frown}}{n}}{\uparrow})\big)
    \:\subseteq\:
    {}^{\omega\hspace{-1pt}}\omega\times{}^{\omega\hspace{-1pt}}\omega\,.$
  \end{itemize}
  Then, for all ${n}\neq{m}\in\omega,$
  \begin{itemize}
  \item [(D)]
    $\mathsurround=0pt
    {R}_{{c},{n}}\cap{R}_{{c},{m}}=\varnothing.$
  \end{itemize}

  To prove (D), suppose on the contrary that ${n}<{m}$ and there is
  ${p}=\langle{p}^{0},{p}^{1}\rangle\in
  {R}_{{c},{n}}\cap{R}_{{c},{m}}.$ We may assume without loss of generality that
  ${p}\in\mathbf{S}_{{a}_{{c}}^{{0}{\frown}}{n}}{\uparrow} \times\mathbf{S}_{{a}_{{c}}^{{1}{\frown}}{n}},$
  so that ${p}^{0}\in\mathbf{S}_{{a}_{{c}}^{{0}{\frown}}{n}}{\uparrow}$
  and ${p}^{1}\in\mathbf{S}_{{a}_{{c}}^{{1}{\frown}}{n}}.$
  Since $\mathbf{S}_{{a}_{{c}}^{{1}{\frown}}{n}}\cap
  \mathbf{S}_{{a}_{{c}}^{{1}{\frown}}{m}}=\varnothing$ (because ${n}\neq{m}$),
  we have ${p}^{1}\nin\mathbf{S}_{{a}_{{c}}^{{1}{\frown}}{m}},$
  so ${p}\nin\mathbf{S}_{{a}_{{c}}^{{0}{\frown}}{m}}{\uparrow} \times\mathbf{S}_{{a}_{{c}}^{{1}{\frown}}{m}}.$
  Then ${p}\in\mathbf{S}_{{a}_{{c}}^{{0}{\frown}}{m}}    \times(\mathbf{S}_{{a}_{{c}}^{{1}{\frown}}{m}}{\uparrow})$
  because ${p}\in{R}_{{c},{m}},$
  therefore
  ${p}^{1}\in\mathbf{S}_{{a}_{{c}}^{{1}{\frown}}{m}}{\uparrow}.$
  This contradicts Remark~\ref{rem.<},
  which says that
  $\mathbf{S}_{{a}_{{c}}^{{1}{\frown}}{n}}\cap
  (\mathbf{S}_{{a}_{{c}}^{{1}{\frown}}{m}}{\uparrow})=\varnothing,$
  so (D) is proved.

  Now, for each ${k}<\lh({b}_{{c}}),$
  there is at most one ${n}\in\omega$ such that
  $\shoot_{{L}}({b}_{{c}}{\upharpoonright}{k})\inn{R}_{{c},{n}}$ ---
  this follows from (D) and (4) of Remark~\ref{rem.inn}.
  Then there is some ${n}_{{c}}\in\omega$ such that, for all ${k}<\lh({b}_{{c}}),$
  $$
  \shoot_{{L}}({b}_{{c}}{\upharpoonright}{k})\:\not\inn\:
  {R}_{{c},{n}_{{c}}}\:=\:
  \big(\mathbf{S}_{{a}_{{c}}^{{0}{\frown}}{n}_{{c}}}{\uparrow} \times\mathbf{S}_{{a}_{{c}}^{{1}{\frown}}{n}_{{c}}}\big)\:\cup\: \big(\mathbf{S}_{{a}_{{c}}^{{0}{\frown}}{n}_{{c}}}    \times(\mathbf{S}_{{a}_{{c}}^{{1}{\frown}}{n}_{{c}}}{\uparrow})\big).
  $$
  Also we have
  $$
  \mathbf{S}_{{c}^{0}}\times\mathbf{S}_{{c}^{1}}\:\supseteq\:
  {L}_{{b}_{{c}}}\:\supseteq\:
  \mathbf{S}_{{a}_{{c}}^{0}}\times\mathbf{S}_{{a}_{{c}}^{1}}
  \:\supseteq\:\mathbf{S}_{{{a}_{{c}}^{{0}{\frown}}{n}_{{c}}}}\times
  \mathbf{S}_{{{a}_{{c}}^{{1}{\frown}}{n}_{{c}}}},
  $$
  so $\sigma^{0}({c})\coloneq{a}_{{c}}^{{0}{\frown}}{n}_{{c}}$ and
  $\sigma^{1}({c})\coloneq{a}_{{c}}^{{1}{\frown}}{n}_{{c}}$ define a strategy $\sigma$ that we search.
\end{proof}

\begin{proof}[\textbf{\textup{Proof of Lemma~\ref{lem.M^i.g^i}}}]
  This lemma says that for each game $\Gamma,$  there are sets ${M}^{0},{M}^{1}\in[\omega]^{\omega}$ and functions ${g}^{0}\colon{M}^{0}\to\omega,$ ${g}^{1}\colon{M}^{1}\to\omega$ that possess the following property:

  \begin{itemize}
  \item[\textup{(\ding{117})}]
    For each candidate ${L}$
    and each ${a}\in{}^{{<}\omega\hspace{-1pt}}\omega,$
    \begin{itemize}
    \item[\ding{226}\,]
      if game $\Gamma$ follows strategy $\sigma_{\hspace{-1pt}\scriptscriptstyle{L}}$ and
      $\,{L}_{{a}}\ni\langle\rr^{0}(\Gamma),\rr^{1}(\Gamma)\rangle,$

    \item[\ding{226}\,]
      then
      $\ \shoot_{{L}}({a})\:\not\inn\:
      \widehat{\mathbf{S}}_{\rr^{0}(\Gamma)}({M}^{0},{g}^{0})\times
      \widehat{\mathbf{S}}_{\rr^{1}(\Gamma)}({M}^{1},{g}^{1}).$
    \end{itemize}
  \end{itemize}
  Suppose that $\Gamma=\big\langle\langle{c}_{{n}}\rangle_{{n}\in\omega},  \langle{d}_{{n}}\rangle_{{n}\in\omega}\big\rangle$ is a game. Put   ${r}^{0}\coloneq\rr^{0}(\Gamma)\in{}^{\omega\hspace{-1pt}}\omega$ and
  ${r}^{1}\coloneq\rr^{1}(\Gamma)\in{}^{\omega\hspace{-1pt}}\omega.$
  Here are sets and functions that we must find:

  \begin{itemize}
  \item[\ding{46}\,]
    $\mathsurround=0pt
    {M}^{0}\:\coloneq\:\{\lh({c}^{0}_{2{k}+1})-1:{k}\in\omega\};$
  \item[\ding{46}\,]
    $\mathsurround=0pt
    {M}^{1}\:\coloneq\:\{\lh({c}^{1}_{2{k}+2})-1:{k}\in\omega\};$
  \item[\ding{46}\,]
    $\mathsurround=0pt
    {g}^{0}\,$ is the function from ${M}^{0}$ to $\omega$ such that ${g}^{0}({m})={r}^{0}({m})+1$ for all
    ${m}\in{M}^{0};$
  \item[\ding{46}\,]
    $\mathsurround=0pt
    {g}^{1}\,$ is the function from ${M}^{1}$ to $\omega$ such that ${g}^{1}({m})={r}^{1}({m})+1$ for all
    ${m}\in{M}^{1}.$
  \end{itemize}
  We must show that (\ding{117}) holds.
  Suppose that ${L}$ is a candidate, ${a}\in{}^{{<}\omega\hspace{-1pt}}\omega,$ game $\Gamma$ follows strategy $\sigma_{\hspace{-1pt}\scriptscriptstyle{L}},$ and
  ${L}_{{a}}\ni\langle{r}^{0},{r}^{1}\rangle;$
  we must prove that $\shoot_{{L}}({a})\:\not\inn\:
  \widehat{\mathbf{S}}_{{r}^{0}}({M}^{0},{g}^{0})\times      \widehat{\mathbf{S}}_{{r}^{1}}({M}^{1},{g}^{1}).$
  Suppose on the contrary that

  \begin{itemize}
  \item[(E1)]
    $\mathsurround=0pt
    \shoot_{{L}}({a})\:\inn\:
    \widehat{\mathbf{S}}_{{r}^{0}}({M}^{0},{g}^{0})\times      \widehat{\mathbf{S}}_{{r}^{1}}({M}^{1},{g}^{1}).$
  \end{itemize}

  Let ${b}_{{-}1}\coloneq\langle\rangle\in{}^{{<}\omega\hspace{-1pt}}\omega.$
  Recall that ${d}^{{i}}_{{n}}=\sigma_{\hspace{-1pt}\scriptscriptstyle{L}}^{i}({c}_{{n}})$ for all ${n}\in\omega$ and ${i}\in{2}$ because game $\Gamma$ follows strategy $\sigma_{\hspace{-1pt}\scriptscriptstyle{L}}.$
  If ${n}\in\omega,$ then, by definition of $\sigma_{\hspace{-1pt}\scriptscriptstyle{L}}$ (see Notaiton~\ref{not.sigma_L}),
  for ${c}_{{n}}=\langle{c}^{0}_{{n}},{c}^{1}_{{n}}\rangle
  \in{}^{{<}\omega\hspace{-1pt}}\omega\times{}^{{<}\omega\hspace{-1pt}}\omega,$
  there is ${b}_{{n}}\in{}^{{<}\omega\hspace{-1pt}}\omega$ such that

  \begin{itemize}
  \item[(E2)]
    $\mathsurround=0pt
    \mathbf{S}_{{c}^{0}_{{n}}}\times\mathbf{S}_{{c}^{1}_{{n}}}\:\supseteq\:
    {L}_{{b}_{{n}}}\:\supseteq\:
    \mathbf{S}_{{d}^{0}_{{n}}}\times
    \mathbf{S}_{{d}^{1}_{{n}}}\quad$ and
  \item[(E3)]
    $\mathsurround=0pt
    \shoot_{{L}}({b}_{{n}}{\upharpoonright}{k})\:\not\inn\:
    \big(\mathbf{S}_{{d}^{0}_{{n}}}{\uparrow}
    \times\mathbf{S}_{{d}^{1}_{{n}}}\big)\cup
    \big(\mathbf{S}_{{d}^{0}_{{n}}}
    \times(\mathbf{S}_{{d}^{1}_{{n}}}{\uparrow})\big)\quad$
    for all ${k}<\lh({b}_{{n}}).$
  \end{itemize}
  We have ${L}_{{b}_{{n}}}\supset{L}_{{b}_{{n}+1}}$ and
  ${L}_{{b}_{{n}}}\ni\langle{r}^{0},{r}^{1}\rangle$ for all ${n}\in\omega$
  by (E2) and (1) of Remark~\ref{rem.game}.
  Then ${b}_{{n}}\sqsubset{b}_{{n}+1}$ for all ${n}\in\omega$
  by (2) of Remark~\ref{rem.st.luz.sch}.
  Also we have ${L}_{{a}}\ni\langle{r}^{0},{r}^{1}\rangle$ by the choice of ${a},$
  so ${L}_{{a}}\cap{L}_{{b}_{{n}}}\neq\varnothing,$
  and hence ${a}\sqsupseteq{b}_{{n}}$ or ${a}\sqsubseteq{b}_{{n}}$ for all ${n}\in\omega$
  by (3) of Remark~\ref{rem.st.luz.sch}.
  Therefore there is $\dot{n}\in\omega$ such that ${b}_{\dot{n}-1}\sqsubseteq{a}\sqsubset{b}_{\dot{n}}.$

  Let ${c}^{0}_{{-}1}\coloneq{c}^{1}_{{-}1}\coloneq\langle\rangle\in{}^{{<}\omega\hspace{-1pt}}\omega.$
  We have $\shoot_{{L}}({a})\inn{L}_{{a}}$ by (1) of Remark~\ref{rem.inn},
  ${L}_{{a}}\subseteq{L}_{{b}_{\dot{n}-1}}$ by the choice of $\dot{n}$ and (2) of Remark~\ref{rem.st.luz.sch}, and
  ${L}_{{b}_{\dot{n}-1}} \subseteq\,\mathbf{S}_{{c}^{0}_{\dot{n}-1}}\!\!\times\mathbf{S}_{{c}^{1}_{\dot{n}-1}}$
  by (E2) (and because ${L}_{{b}_{-1}}={}^{\omega\hspace{-1pt}}\omega\times{}^{\omega\hspace{-1pt}}\omega
  =\mathbf{S}_{{c}^{0}_{-1}}\times\mathbf{S}_{{c}^{1}_{-1}}$
  by (L2) of Definition~\ref{def.luz.p.base}), so $\shoot_{{L}}({a})\inn
  \mathbf{S}_{{c}^{0}_{\dot{n}-1}}\!\!\times\mathbf{S}_{{c}^{1}_{\dot{n}-1}}$
  by (2) of Remark~\ref{rem.inn}.
  Therefore
  $$
  \shoot_{{L}}({a})\:\inn\:
  \big(\,\widehat{\mathbf{S}}_{{r}^{0}}({M}^{0},{g}^{0})\times      \widehat{\mathbf{S}}_{{r}^{1}}({M}^{1},{g}^{1})\big)
  \,\cap\,
  (\,\mathbf{S}_{{c}^{0}_{\dot{n}-1}}\!\!\times\mathbf{S}_{{c}^{1}_{\dot{n}-1}})
  $$
  by (E1) and (3) of Remark~\ref{rem.inn}, and then using (2) of Remark~\ref{rem.inn} we get
  $$
  \shoot_{{L}}({a})\:\inn\:
  \big(\,\widehat{\mathbf{S}}_{{r}^{0}}({M}^{0},{g}^{0})\cap\mathbf{S}_{{c}^{0}_{\dot{n}-1}}\big)
  \,\times\,
  \big(\,\widehat{\mathbf{S}}_{{r}^{1}}({M}^{1},{g}^{1})\cap\mathbf{S}_{{c}^{1}_{\dot{n}-1}}\big).
  $$
  Let

  \begin{itemize}
  \item[\ding{46}\,]
    $\mathsurround=0pt
    {T}^{0}\:\coloneq\:\widehat{\mathbf{S}}_{{r}^{0}}({M}^{0},{g}^{0})\cap \mathbf{S}_{{c}^{0}_{\dot{n}-1}}$ and
  \item[\ding{46}\,]
    $\mathsurround=0pt
    {T}^{1}\:\coloneq\:
  \widehat{\mathbf{S}}_{{r}^{1}}({M}^{1},{g}^{1})\cap
  \mathbf{S}_{{c}^{1}_{\dot{n}-1}},$
  \end{itemize}
  so that

  \begin{itemize}
  \item[(E4)]
    $\mathsurround=0pt
    \shoot_{{L}}({a})\:\inn\:{T}^{0}\times{T}^{1}.$
  \end{itemize}

  Next we show the following:

  \begin{itemize}
  \item[($\mathsurround=0pt \text{E5}^0$)]
    if $\dot{n}$ is even, then
    ${T}^{0}\subseteq\mathbf{S}_{{d}^{0}_{\dot{n}}}\quad$ and
  \item[($\mathsurround=0pt \text{E5}^1$)]
    if $\dot{n}$ is odd, then
    ${T}^{1}\subseteq\mathbf{S}_{{d}^{1}_{\dot{n}}}.$
  \end{itemize}
  To prove ($\mathsurround=0pt \text{E5}^0$), assume that $\dot{n}$ is even and ${q}\in{T}^{0}.$
  Recall that
  $\widehat{\mathbf{S}}_{{r}^{0}}({M}^{0},{g}^{0})=    \{{r}^{0}\}\cup\bigcup_{{m}\in{M}^{0}}
  \widetilde{\mathbf{S}}^{{g}^{0}({m})}_{{p}^{0}{\upharpoonright}{m}}.$
  If ${q}={r}^{0},$ then ${q}\in\mathbf{S}_{{d}^{0}_{\dot{n}}}$
  by (1) of Remark~\ref{rem.game}.
  If ${q}\neq{r}^{0},$ then there is $\dot{m}\in{M}^{0}$ such that
  ${q}\in\widetilde{\mathbf{S}}^{{g}^{0}(\dot{m})}_{{r}^{0}{\upharpoonright}\dot{m}},$
  so

  \begin{itemize}
  \item[($\mathsurround=0pt \text{E6}^0$)]
    $\mathsurround=0pt
    {q}\in\widetilde{\mathbf{S}}^{{r}^{0}(\dot{m})+1}_{{r}^{0}{\upharpoonright}\dot{m}}$
  \end{itemize}
  by the choice of ${g}^{0}.$
  Since $\dot{m}\in{M}^{0},$ there is $\dot{k}\in\omega$ such that

  \begin{itemize}
  \item[($\mathsurround=0pt \text{E7}^0$)]
    $\mathsurround=0pt
    \dot{m}=\lh({c}^{0}_{2\dot{k}+1})-1.$
  \end{itemize}
  We have
  ${q}\nin{\mathbf{S}}_{{r}^{0}{\upharpoonright}(\dot{m}+1)}$
  by ($\mathsurround=0pt \text{E6}^0$) and (1) of Remark~\ref{rem.lhd}, so
  ${q}\nin{\mathbf{S}}_{{r}^{0}{\upharpoonright}\lh({c}^{0}_{2\dot{k}+1})}
  ={\mathbf{S}}_{{c}^{0}_{2\dot{k}+1}}$ by ($\mathsurround=0pt \text{E7}^0$) and (2) of Remark~\ref{rem.game}.
  On the other hand, ${q}\in{\mathbf{S}}_{{c}^{0}_{\dot{n}-1}}$
  by definition of ${T}^{0},$  therefore
  ${\mathbf{S}}_{{c}^{0}_{2\dot{k}+1}}\nsupseteq{\mathbf{S}}_{{c}^{0}_{\dot{n}-1}},$
  hence
  ${c}^{0}_{2\dot{k}+1}\nsqsubseteq{c}^{0}_{\dot{n}-1}$ by (2) of Remark~\ref{rem.st.luz.sch},
  so $2\dot{k}+1\nleqslant\dot{n}-1$ by (5) of Remark~\ref{rem.game},
  that is, $2\dot{k}+1>\dot{n}-1.$
  Then $2\dot{k}\geqslant\dot{n}-1$
  and also $2\dot{k}\neq\dot{n}-1$ because $\dot{n}$ is even, so we get

  \begin{itemize}
  \item[($\mathsurround=0pt \text{E8}^0$)]
    $\mathsurround=0pt
    2\dot{k}\geqslant\dot{n}.$
  \end{itemize}
  Now,  ${q}\in
  \widetilde{\mathbf{S}}^{{r}^{0}(\dot{m})+1}_{{r}^{0}{\upharpoonright}\dot{m}}$ by ($\mathsurround=0pt \text{E6}^0$),
  $\widetilde{\mathbf{S}}^{{r}^{0}(\dot{m})+1}_{{r}^{0}{\upharpoonright}\dot{m}}
  \subseteq\widetilde{\mathbf{S}}_{{r}^{0}{\upharpoonright}\dot{m}}$
  by (1) of Remark~\ref{rem.lhd},
  $\widetilde{\mathbf{S}}_{{r}^{0}{\upharpoonright}\dot{m}}=
  {\mathbf{S}}_{{r}^{0}{\upharpoonright}(\lh({c}^{0}_{2\dot{k}+1})-1)}$
  by ($\mathsurround=0pt \text{E7}^0$),
  ${\mathbf{S}}_{{r}^{0}{\upharpoonright}(\lh({c}^{0}_{2\dot{k}+1})-1)}
  \subseteq{\mathbf{S}}_{{r}^{0}{\upharpoonright}\lh({d}^{0}_{2\dot{k}})}$
  by (4) of Remark~\ref{rem.game} and (2) of Remark~\ref{rem.st.luz.sch},
  ${\mathbf{S}}_{{r}^{0}{\upharpoonright}\lh({d}^{0}_{2\dot{k}})}
  ={\mathbf{S}}_{{d}^{0}_{2\dot{k}}}$
  by (3) of Remark~\ref{rem.game}, and
  ${\mathbf{S}}_{{d}^{0}_{2\dot{k}}}
  \subseteq{\mathbf{S}}_{{d}^{0}_{\dot{n}}}$
  by ($\mathsurround=0pt \text{E8}^0$), (6) of Remark~\ref{rem.game}, and (2) of Remark~\ref{rem.st.luz.sch}.
  Therefore ${q}\in\mathbf{S}_{{d}^{0}_{\dot{n}}},$ which proves ($\mathsurround=0pt \text{E5}^0$).

  The proof of ($\mathsurround=0pt \text{E5}^1$) is similar: Assume that $\dot{n}$ is odd and ${q}\in{T}^{1}.$ $\widehat{\mathbf{S}}_{{r}^{1}}({M}^{1},{g}^{1})=    \{{r}^{1}\}\cup\bigcup_{{m}\in{M}^{1}}
  \widetilde{\mathbf{S}}^{{g}^{1}({m})}_{{p}^{1}{\upharpoonright}{m}}.$
  If ${q}={r}^{1},$ then ${q}\in\mathbf{S}_{{d}^{1}_{\dot{n}}}.$
  If ${q}\neq{r}^{1},$ then there is $\dot{m}\in{M}^{1}$ such that
  ${q}\in\widetilde{\mathbf{S}}^{{g}^{1}(\dot{m})}_{{r}^{1}{\upharpoonright}\dot{m}},$
  so

  \begin{itemize}
  \item[($\mathsurround=0pt \text{E6}^1$)]
    $\mathsurround=0pt
    {q}\in\widetilde{\mathbf{S}}^{{r}^{1}(\dot{m})+1}_{{r}^{1}{\upharpoonright}\dot{m}}\:.$
  \end{itemize}
  Since $\dot{m}\in{M}^{1},$ there is $\dot{k}\in\omega$ such that

  \begin{itemize}
  \item[($\mathsurround=0pt \text{E7}^1$)]
    $\mathsurround=0pt
    \dot{m}=\lh({c}^{1}_{2\dot{k}+2})-1.$
  \end{itemize}
  We have
  ${q}\nin{\mathbf{S}}_{{r}^{1}{\upharpoonright}(\dot{m}+1)}=
  {\mathbf{S}}_{{r}^{1}{\upharpoonright}\lh({c}^{1}_{2\dot{k}+2})}
  ={\mathbf{S}}_{{c}^{1}_{2\dot{k}+2}}.$
  On the other hand, ${q}\in{\mathbf{S}}_{{c}^{1}_{\dot{n}-1}},$
  therefore
  ${\mathbf{S}}_{{c}^{1}_{2\dot{k}+2}}\nsupseteq{\mathbf{S}}_{{c}^{1}_{\dot{n}-1}},$
  hence
  ${c}^{1}_{2\dot{k}+2}\nsqsubseteq{c}^{1}_{\dot{n}-1},$
  so $2\dot{k}+2\nleqslant\dot{n}-1,$
  that is, $2\dot{k}+2>\dot{n}-1.$
  Then $2\dot{k}+1\geqslant\dot{n}-1$
  and also $2\dot{k}+1\neq\dot{n}-1$ because $\dot{n}$ is odd, so we get

  \begin{itemize}
  \item[($\mathsurround=0pt \text{E8}^1$)]
    $\mathsurround=0pt
    2\dot{k}+1\geqslant\dot{n}.$
  \end{itemize}
  By the same argument as above, we can write
  $$
  {q}\in\widetilde{\mathbf{S}}^{{r}^{1}(\dot{m})+1}_{{r}^{1}{\upharpoonright}\dot{m}}
  \:\subseteq\:
  \widetilde{\mathbf{S}}_{{r}^{1}{\upharpoonright}\dot{m}}
  \:=\:
  {\mathbf{S}}_{{r}^{1}{\upharpoonright}(\lh({c}^{1}_{2\dot{k}+2})-1)}
  \:\subseteq\:
  {\mathbf{S}}_{{r}^{1}{\upharpoonright}\lh({d}^{1}_{2\dot{k}+1})}
  \:=\:
  {\mathbf{S}}_{{d}^{1}_{2\dot{k}+1}}
  \:\subseteq\:
  {\mathbf{S}}_{{d}^{1}_{\dot{n}}}.
  $$
  This proves ($\mathsurround=0pt \text{E5}^1$).

  Now, ($\mathsurround=0pt \text{E5}^0$) and ($\mathsurround=0pt \text{E5}^1$) imply

  \begin{itemize}
  \item[(E5)]
    $\mathsurround=0pt
    {T}^{0}\subseteq\mathbf{S}_{{d}^{0}_{\dot{n}}}\quad\mathsf{or}\quad
    {T}^{1}\subseteq\mathbf{S}_{{d}^{1}_{\dot{n}}}.$
  \end{itemize}
  Also, for each ${i}\in{2},$ we have
  ${T}^{{i}}\subseteq\widehat{\mathbf{S}}_{{r}^{{i}}}({M}^{{i}},{g}^{{i}})$
  by definition of ${T}^{{i}},$
  $\widehat{\mathbf{S}}_{{r}^{i}}({M}^{i},{g}^{i})    \subseteq{r}^{{i}}{\downfilledspoon}$
  by (2) of Remark~\ref{rem.lhd}, and
  ${r}^{{{i}}}{\downfilledspoon}\,\subseteq\,
  \mathbf{S}_{{d}^{{{i}}}_{\dot{n}}}{\uparrow}$
  (because ${r}^{{{i}}}\in\mathbf{S}_{{d}^{{{i}}}_{\dot{n}}}$
  by (1) of Remark~\ref{rem.game}), so

  \begin{itemize}
  \item[(E9)]
    $\mathsurround=0pt
    {T}^{0}\,\subseteq\,\mathbf{S}_{{d}^{0}_{\dot{n}}}{\uparrow}\quad\mathsf{and}\quad
    {T}^{1}\,\subseteq\,\mathbf{S}_{{d}^{1}_{\dot{n}}}{\uparrow}.$
  \end{itemize}
  It follows from (E5) and (E9) that
  ${T}^{0}\times{T}^{1}\subseteq
  \mathbf{S}_{{d}^{0}_{\dot{n}}}\times
  (\mathbf{S}_{{d}^{1}_{\dot{n}}}{\uparrow})$ or
  ${T}^{0}\times{T}^{1}\subseteq
  \mathbf{S}_{{d}^{0}_{\dot{n}}}{\uparrow}\times
  \mathbf{S}_{{d}^{1}_{\dot{n}}},$
  so
  $$
  {T}^{0}\times{T}^{1}\:\subseteq\:\big(\mathbf{S}_{{d}^{0}_{\dot{n}}}{\uparrow}\times
  \mathbf{S}_{{d}^{1}_{\dot{n}}}\big)\cup
  \big(\mathbf{S}_{{d}^{0}_{\dot{n}}}\times
  (\mathbf{S}_{{d}^{1}_{\dot{n}}}{\uparrow})\big).
  $$
  Then, using (E4) and (2) of Remark~\ref{rem.inn}, we get
  $$
  \shoot_{{L}}({a})\:\inn\:
  \big(\mathbf{S}_{{d}^{0}_{\dot{n}}}{\uparrow}\times
  \mathbf{S}_{{d}^{1}_{\dot{n}}}\big)\cup
  \big(\mathbf{S}_{{d}^{0}_{\dot{n}}}\times
  (\mathbf{S}_{{d}^{1}_{\dot{n}}}{\uparrow})\big).
  $$
  But also, since ${a}\sqsubset{b}_{\dot{n}}$ by the choice of $\dot{n},$
  there is ${k}<\lh({b}_{\dot{n}})$ such that
  ${a}={b}_{\dot{n}}{\upharpoonright}{k}.$ Then
  $\shoot_{{L}}({a})=\shoot_{{L}}({b}_{\dot{n}}{\upharpoonright}{k})$
  and we have a contradiction with (E3), which says
  $$
  \shoot_{{L}}({b}_{\dot{n}}{\upharpoonright}{k})
  \:\not\inn\:
  \big(\mathbf{S}_{{d}^{0}_{\dot{n}}}{\uparrow}
  \times\mathbf{S}_{{d}^{1}_{\dot{n}}}\big)\cup
  \big(\mathbf{S}_{{d}^{0}_{\dot{n}}}
  \times(\mathbf{S}_{{d}^{1}_{\dot{n}}}{\uparrow})\big).
  $$
\end{proof}

%=================������ ����������====================

\end{document}